\documentclass[twoside,reqno]{amsart}

\usepackage{amsmath,amssymb,amsbsy,amstext,amsxtra,latexsym}
\usepackage{graphicx}
\usepackage{subfig}
\usepackage{enumerate}
\usepackage[all]{xy}
\usepackage{url}
\usepackage{cancel}

\usepackage{mathptmx}       
\usepackage{helvet}         
\usepackage{courier}        

\usepackage{hyperref}

\hypersetup{%
pdftitle={A simply connected numerical Campedelli surface with an involution},%
pdfauthor={Heesang Park, Dongsoo Shin, Giancarlo Urzua},%
pdfkeywords={surface of general type, involution, $\mathbb{Q}$-Gorenstein smoothing},%
}

\newtheorem{theorem}{Theorem}[section]
\newtheorem*{maintheorem}{Theorem}
\newtheorem{lemma}[theorem]{Lemma}
\newtheorem{proposition}[theorem]{Proposition}
\newtheorem{corollary}[theorem]{Corollary}

\theoremstyle{definition}

\theoremstyle{remark}

\newtheorem{remark}[theorem]{Remark}

\numberwithin{equation}{section}

\newcommand{\uc}[1]{\ensuremath \overset{#1}{\circ}}
\newcommand{\udc}[2]{\ensuremath \underset{#2}{\overset{#1}{\circ}}}

\newcommand{\lineqv}{\ensuremath \sim}

\newcommand{\abs}[1]{\ensuremath \lvert #1 \rvert}

\DeclareMathOperator{\pic}{Pic}

\DeclareMathOperator{\Defq}{Def^{\text{QG}}}

\def\sheaf#1{\ensuremath \mathcal#1}

\begin{document}

\title[A numerical Campedelli surface with an involution]{A simply connected numerical Campedelli surface \\ with an involution}

\author[H. Park]{Heesang Park}

\address{School of Mathematics, Korea Institute for Advanced Study, Seoul 130-722, Korea}

\email{hspark@kias.re.kr}

\author[D. Shin]{Dongsoo Shin}

\address{Department of Mathematics, Chungnam National University, Daejeon 305-764, Korea \& School of Mathematics, Korea Institute for Advanced Study, Seoul 130-722, Korea}

\email{dsshin@cnu.ac.kr}

\author[G. Urz\'ua]{Giancarlo Urz\'ua}

\address{Facultad de Matem\'aticas, Pontificia Universidad Cat\'olica de Chile, Santiago, Chile.}

\email{urzua@mat.puc.cl}

\date{December 31, 2011; Revised at October 03, 2012}



\begin{abstract}
We construct a simply connected minimal complex surface of general type with $p_g=0$ and $K^2=2$ which has an involution such that the minimal resolution of the quotient by the involution is a simply connected minimal complex surface of general type with $p_g=0$ and $K^2=1$. In order to construct the example, we combine a double covering and $\mathbb{Q}$-Gorenstein deformation. Especially, we develop a method for proving unobstructedness for deformations of a singular surface by generalizing a result of Burns and Wahl which characterizes the space of first order deformations of a singular surface with only rational double points. We describe the stable model in the sense of Koll\'ar and Shepherd-Barron of the singular surfaces used for constructing the example. We count the dimension of the invariant part of the deformation space of the example under the induced $\mathbb{Z}/2\mathbb{Z}$-action.
\end{abstract}

\maketitle


\section{Introduction}

One of the fundamental problems in the classification of complex
surfaces is to find a new family of complex surfaces of general
type with $p_g=0$. In this paper we construct new simply connected
\emph{numerical Campedelli surfaces} with an involution, i.e.
simply connected minimal complex surfaces of general type with
$p_g=0$ and $K^2=2$, that have an automorphism of order $2$.

There has been a growing interest for complex surfaces of general
type with $p_g=0$ having an involution; cf.~J. Keum-Y.
Lee~\cite{Keum-Lee}, Calabri-Ciliberto-Mendes
Lopes~\cite{Calabri-Ciliberto-Lopes}, Calabri-Mendes
Lopes-Pardini~\cite{Calabri-Lopes-Pardini}, Y. Lee-Y.
Shin~\cite{Lee-Shin}, Rito~\cite{Rito}. A classification of
\emph{numerical Godeaux surfaces} (i.e. minimal complex surfaces
of general type with $p_g=0$ and $K^2=1$) with an involution is
given in Calabri-Ciliberto-Mendes
Lopes~\cite{Calabri-Ciliberto-Lopes}. It is known that the
quotient surface of a numerical Godeaux surface by its involution
is either rational or birational to an Enriques surface, and the
bicanonical map of the numerical Godeaux surface factors through
the quotient map.

However, the situation is more involved in the case of numerical
Campedelli surfaces, because the bicanonical map may not factor
through the quotient map; cf.~Calabri-Mendes
Lopes-Pardini~\cite{Calabri-Lopes-Pardini}. In particular it can
happen that the quotient is of general type. More precisely, let
$X$ be a numerical Campedelli surface with an involution $\sigma$.
If $\sigma$ has only fixed points and no fixed divisors, then the
minimal resolution $S$ of the quotient $Y=X/\sigma$ is a numerical
Godeaux surface and $\sigma$ has only four fixed points;
cf.~Barlow~\cite{Barlow-Invent}. Conversely, if $S$ is of general
type, then $\sigma$ has only four fixed points and no fixed
divisors; Calabri-Mendes
Lopes-Pardini~\cite{Calabri-Lopes-Pardini}.

There are some examples of numerical Campedelli surfaces $X$ with
an involution $\sigma$ having only four fixed points.
Barlow~\cite{Barlow-Duke} constructed examples with
$\pi_1(X)=\mathbb{Z}/2\mathbb{Z} \oplus \mathbb{Z}/4\mathbb{Z},
\mathbb{Z}/8\mathbb{Z}$. Barlow~\cite{Barlow-Invent} also
constructed examples with $\pi_1(X)=\mathbb{Z}/5\mathbb{Z}$ whose
minimal resolution of the quotient by the involution is the first
example of a \emph{simply connected} numerical Godeaux surface.
Also all Catanese's surfaces \cite{Catanese} have such an
involution and $\pi_1=\mathbb{Z}/5\mathbb{Z}$. Recently Calabri,
Mendes Lopes, and Pardini~\cite{Calabri-Lopes-Pardini} constructed
a numerical Campedelli surface with torsion
$\mathbb{Z}/3\mathbb{Z} \oplus \mathbb{Z}/3\mathbb{Z}$ and two
involutions. Frapporti~\cite{Frapporti} showed that there exists an involution having only four fixed points on the numerical Campedelli surface with $\mathbb{Z}/2\mathbb{Z} \oplus \mathbb{Z}/4\mathbb{Z}$ constructed first in Bauer-Catanese-Grunewald-Pignatelli~\cite{Bauer-Catanese-Grunewald-Pignatelli}.

It is known that the orders of the algebraic fundamental groups of
numerical Campedelli surfaces are at most $9$ and the dihedral
groups $D_3$ and $D_4$ cannot be realized. Recently, the existence
question for numerical Campedelli surfaces with
$\abs{\pi_1^{\text{alg}}} \le 9$ was settled by the construction
of examples with $\pi_1^{\text{alg}}=\mathbb{Z}/4\mathbb{Z}$;
Frapporti~\cite{Frapporti} and H. Park-J. Park-D.
Shin~\cite{PPS-K2-H1Z4}. Hence it would be an interesting problem
to construct numerical Campedelli surfaces having an involution
with $\pi_1^{\text{alg}}=G$ for each given group $G$ with $\abs{G}
\le 9$. Especially we are concerned with the simply connected case
because the fundamental groups of all the known examples with an
involution have large order: $\abs{G} \ge 5$. Furthermore the
first example of \emph{simply connected} numerical Campedelli
surfaces is very recent (Y. Lee-J. Park~\cite{Lee-Park-K^2=2}),
but we have no information about the existence of an involution in
their example. The main theorem of this paper is:

\begin{maintheorem}[Corollary~\ref{corollary:Campedelli-with-involution}]
There are simply connected minimal complex surfaces $X$ of general
type with $p_g(X)=0$ and $K_X^2=2$ which have an involution
$\sigma$ such that the minimal resolution $S$ of the quotient
$Y=X/\sigma$ is a simply connected minimal complex surface of
general type with $p_g(S)=0$ and $K_S^2=1$.
\end{maintheorem}

We also show that the minimal resolution $S$ of the quotient
$Y=X/\sigma$ has a local deformation space of dimension $4$
corresponding to deformations $\mathcal{S}$ of $S$ such that its
general fiber $\mathcal{S}_t$ is the minimal resolution of a
quotient $X_t/\sigma_t$ of a numerical Campedelli surface $X_t$ by
an involution $\sigma_t$; Theorem~\ref{theorem:invariant-part}. In
addition, we show that the resolution $S$ should be always simply
connected if the double cover $X$ is already simply connected;
Proposition~\ref{proposition:simply-connected=>simply-connected}.
Conversely Barlow~\cite{Barlow-Invent} showed that if the
resolution $S$ is a simply connected numerical Godeaux surface
then the possible order of the algebraic fundamental group of the
double cover $X$ is $1$, $3$, $5$, $7$, or $9$. As far as we know,
the example in Barlow~\cite{Barlow-Invent} was the only one whose
quotient is simply connected. It has $\pi_1(X) =
\mathbb{Z}/5\mathbb{Z}$ as mentioned earlier. Here we find an
example with $\pi_1(X)=1$. Hence it would be an intriguing problem
in this context to construct an example with
$\pi_1(X)=\mathbb{Z}/3\mathbb{Z}$.

In order to construct the examples, we combine a double covering
and a $\mathbb{Q}$-Gorenstein smoothing method developed in Y.
Lee-J. Park~\cite{Lee-Park-K^2=2}. First we build singular
surfaces by blowing up points and then contracting curves over a
specific rational elliptic surface. These singular surfaces differ
by contracting certain $(-2)$-curves. If we contract all of the
$(-2)$-curves, we obtain a stable surface $Y'$ in the sense of
Koll\'ar--Shepherd-Barron~\cite{KSB88}, and we prove that the
space of $\mathbb{Q}$-Gorenstein deformations of $Y'$ is smooth
and $8$ dimensional; Proposition~\ref{propsotion:stable-godeaux}.
A ($\mathbb{Q}$-Gorenstein) smoothing of $Y'$ in this space
produces simply connected numerical Godeaux surfaces. In
particular, the smoothing of $Y'$ gives the existence of a two
dimensional family of simply connected numerical Godeaux surfaces
with six $(-2)$-curves; Corollary~\ref{Corollary:six}. We also
prove that a four dimensional family in this space produces simply
connected numerical Godeaux surfaces with a $2$-divisible divisor
consisting of four disjoint $(-2)$-curves;
Theorem~\ref{theorem:Q-smoothing-of-Y} and
Theorem~\ref{theorem:invariant-part}. These numerical Godeaux
surfaces are used to construct the numerical Campedelli surfaces
with an involution. The desired numerical Campedelli surfaces are
obtained by taking double coverings of the numerical Godeaux
surfaces branched along the four disjoint $(-2)$-curves;
Theorem~\ref{theorem:campedelli}.

On the other hand we can also obtain the Campedelli family
explicitly from a singular stable surface $X'$. It comes from
blowing up points and contracting curves over a certain rational
elliptic surface; Proposition~\ref{proposition:CampExplicit}. The
$\mathbb{Q}$-Gorenstein space of deformations of $X'$ is smooth
and $6$ dimensional;
Proposition~\ref{proposition:stable-campedelli}. In both Godeaux
and Campedelli cases we compute $H^2(\sheaf{T})=0$ to show no
local-to-global obstruction to deform them;
Theorem~\ref{theorem:H2(TY)=0} and Theorem~\ref{theorem:H2(TX)=0}.
This involves a new technique (Theorem~\ref{theorem:Burns-Wahl})
which generalizes a result of Burns-Wahl~\cite{Burns-Wahl}
describing the space of first order deformations of a singular
complex surface with only rational double points.

\subsection*{Notations}

A cyclic quotient singularity (germ at $(0,0)$ of) $\mathbb{C}^2
/G$, where $G=\langle (x,y) \mapsto (\zeta x,\zeta^q y) \rangle$
with $\zeta$ a $m$-th primitive root of $1$, $1<q<m$, and
$(q,m)=1$, is denoted by $\frac{1}{m}(1,q)$. $A_m$ means
$\frac{1}{m}(1,m-1)$. A $(-1)$-curve (or $(-2)$-curve) in a smooth
surface is an embedded $\mathbb{CP}^1$ with self-intersection $-1$
(respectively, $-2$). Throughout this paper we use the same letter
to denote a curve and its proper transform under a birational map.
A singularity of class $T$  is a quotient singularity which admits
a $\mathbb{Q}$-Gorenstein one parameter smoothing. They are either
rational double points or $\frac{1}{dn^2}(1,dna-1)$ with $1<a<n$
and $(n,a)=1$; see Koll\'ar--Shepherd-Barron~\cite[\S3]{KSB88}.
For a normal variety $X$ its tangent sheaf $\sheaf{T_X}$ is
$\sheaf{Hom}_{\sheaf{O_X}}(\Omega_X^1,\sheaf{O_X})$. The dimension
of $H^i$ is $h^i$.

\subsection*{Acknowledgements}

The authors would like to thank Professor Yongnam Lee for helpful
discussion during the work, careful reading of the draft version,
and many valuable comments. The authors also wish to thank
Professor Jenia Tevelev for indicating a mistake in an earlier
version of this paper, and the referee especially for the remark on the proof of Proposition 3.7 which makes it simpler. Heesang park was supported by Basic Science
Research Program through the National Research Foundation of Korea
(NRF) grant funded by the Korean Government (2011-0012111).
Dongsoo Shin was supported by Basic Science Research Program
through the National Research Foundation of Korea (NRF) grant
funded by the Korean Government (2010-0002678). Giancarlo Urz\'ua
was supported by a FONDECYT Inicio grant funded by the Chilean
Government (11110047).

\section{Numerical Godeaux surfaces with a $2$-divisible divisor}
\label{section:Godeaux}

In this section we construct a family of simply connected
numerical Godeaux surfaces having a $2$-divisible divisor consisting of
four disjoint $(-2)$-curves by smoothing a singular surface $\widetilde{Y}$; Theorem~\ref{theorem:Q-smoothing-of-Y}. This is the key to construct numerical Campedelli surfaces with an involution. In addition, we describe the explicit stable model of the singular surface $\widetilde{Y}$. In fact, we construct a rational normal projective surface
$Y'$ with four singularities $A_3$, $A_3$, $\frac{1}{8^2}(1,8\cdot
5-1)$, $\frac{1}{7^2}(1,7 \cdot 4-1)$ and $K_{Y'}$ ample. Hence $Y'$
is a stable surface (cf. Koll\'ar-Shepherd-Barron~\cite{KSB88}, Hacking~\cite{Hacking}). We will
prove that the versal $\mathbb{Q}$-Gorenstein deformation space $\Defq
(Y')$ (cf.~Hacking~\cite[\S3]{Hacking}) is smooth and $8$
dimensional, and that the $\mathbb{Q}$-Gorenstein smoothings of $Y'$ are
simply connected numerical Godeaux surfaces. In particular, this shows that there are simply
connected numerical Godeaux surfaces whose canonical model has
precisely two $A_3$ singularities; Corollary~\ref{corollary:Godeaux}. Furthermore a four dimensional family in $\Defq(Y')$ produces the above simply connected numerical Godeaux surfaces with a $2$-divisible divisor consisting of four disjoint $(-2)$-curves; Theorem~\ref{theorem:Q-smoothing-of-Y} and Theorem~\ref{theorem:invariant-part}.

\subsection{A rational elliptic surface $E(1)$}

We start with a rational elliptic surface $E(1)$ with an
$I_8$-singular fiber, an $I_2$-singular fiber, and two nodal
singular fibers. In fact we will use the same rational elliptic
surface $E(1)$ in the papers H. Park-J. Park-D. Shin~\cite{PPS-K3,
PPS-K4}. However, we need to sketch the construction of $E(1)$ to
show the relevant curves that will be used to build the singular surfaces $\widetilde{Y}$ and $Y'$.

Let $L_1$, $L_2$, $L_3$, and $\ell$ be lines in $\mathbb{CP}^2$
and let $c$ be a smooth conic in $\mathbb{CP}^2$ given by the
following equations. They intersect as in
Figure~\ref{figure:pencil}.
\begin{equation*}
\begin{aligned}
&L_1: 2y-3z=0, \quad L_2: y+\sqrt{3}x=0, \quad L_3: y-\sqrt{3}x=0 \\
&\ell: x=0, \quad c: x^2 + (y-2z)^2-z^2=0.
\end{aligned}
\end{equation*}

\begin{figure}[tbh]
\centering
\includegraphics[scale=1]{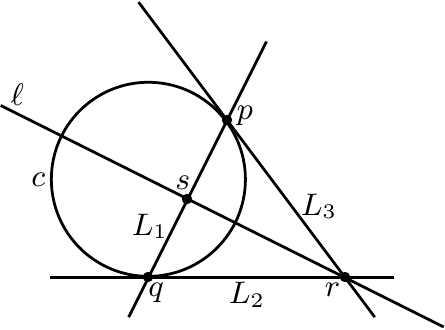}
\caption{A pencil of cubic curves}
\label{figure:pencil}
\end{figure}

We consider the pencil of cubics
\begin{equation*}
\{\lambda (y-\sqrt{3}x)(y+\sqrt{3}x)(2y-3z) + \mu x(x^2 + (y-2z)^2-z^2) \mid [\lambda:\mu] \in \mathbb{CP}^1 \}
\end{equation*}
generated by the two cubic curves $L_1+L_2+L_3$ and $\ell+c$. This
pencil has four base points $p$, $q$, $r$, $s$, and four singular
members corresponding to $[\lambda:\mu]=[1:0], [0:1],
[2:3\sqrt{3}], [2:-3\sqrt{3}]$. The latter two singular members
are nodal curves, denoted by $F_1$ and $F_2$ respectively. They
have nodes at $[-\sqrt{3}:0:1]$ and $[\sqrt{3}:0:1]$,
respectively.

In order to obtain a rational elliptic surface $E(1)$ from the
pencil, we resolve all base points (including infinitely near
base-points) of the pencil by blowing-up $9$ times as follows. We
first blow up at the points $p$, $q$, $r$, $s$. Let $e_1$, $e_2$,
$e_3$, $e_4$ be the exceptional divisors over $p$, $q$, $r$, $s$,
respectively. We blow up again at the three points $e_1 \cap L_3$,
$e_2 \cap L_2$, $e_3 \cap \ell$. Let $e_5$, $e_6$, $e_7$ be the
exceptional divisors over the intersection points, respectively.
We finally blow up at each intersection points $e_5 \cap c$ and
$e_6 \cap c$. Let $e_8$ and $e_9$ be the exceptional divisors over
the blown-up points. We then get a rational elliptic surface
$E(1)=\mathbb{CP}^2\sharp 9\overline{\mathbb{CP}}^2$ over
$\mathbb{CP}^1$; see Figure~\ref{figure:E}.

The four exceptional curves $e_4$, $e_7$, $e_8$, $e_9$ are
sections of the elliptic fibration $E(1)$, which correspond to the
four base points $s$, $r$, $p$, $q$, respectively. The elliptic
fibration $E(1)$ has one $I_8$-singular fiber $\sum_{i=1}^{8} B_i$
containing all $L_i$ ($i=1,2,3$): $B_2=L_2$, $B_3=L_3$, and
$B_5=L_1$; cf.~Figure~\ref{figure:E}. We will use frequently the
sum $B=B_1+B_2+B_3+B_4$, which will be shown to be $2$-divisible.
The surface $E(1)$ has also one $I_2$-singular fiber consisting of
$\ell$ and $c$, and it has two more nodal singular fibers $F_1$
and $F_2$.

\begin{figure}[tbh]
\centering
\includegraphics[scale=1]{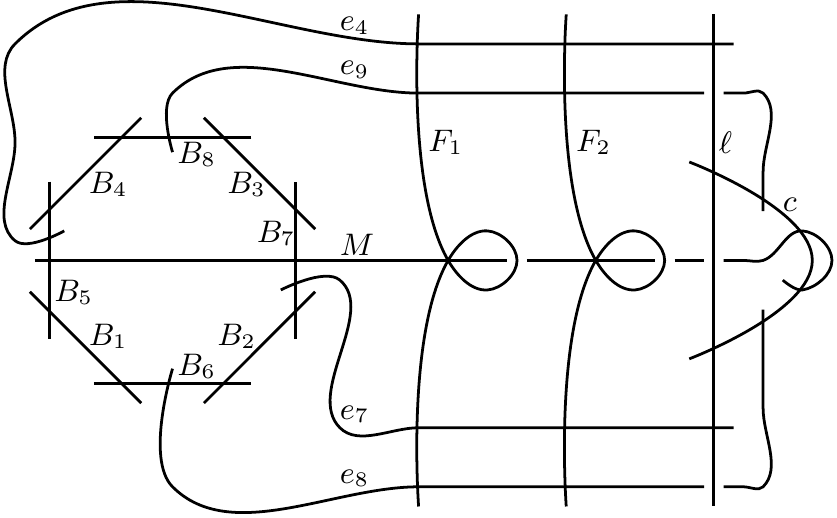}
\caption{A rational elliptic surface $E(1)$}
\label{figure:E}
\end{figure}

There is a special bisection on $E(1)$. Let $M$ be the line in
$\mathbb{CP}^2$ passing through the point $r=[0:0:1]$ and the two
nodes $[-\sqrt{3}:0:1]$ and $[\sqrt{3}:0:1]$ of $F_1$ and $F_2$.
Since $M$ meets every member in the pencil at three points but it
passes through only one base point, the proper transform of $M$ is
a bisection of the elliptic fibration $E(1) \to \mathbb{CP}^1$;
cf.~Figure~\ref{figure:E}. Note that $M^2=0$ in $E(1)$.

\subsubsection*{A $2$-divisible divisor on $E(1)$}

Let $h \in \pic (E(1))$ be the class of the pull-back of a line in
$\mathbb{CP}^2$. We denote again by $e_i \in \pic(E(1))$ the class
of the pull-back of the exceptional divisor $e_i$. We have the
following linear equivalences of divisors in $E(1)$:
\begin{equation*}
\begin{aligned}
B_1 &\lineqv e_2-e_6, & B_2 &\lineqv h-e_2-e_3-e_6, & B_3 &\lineqv h-e_1-e_3-e_5, \\
B_4 &\lineqv e_1-e_5, & B_5 &\lineqv h-e_1-e_2-e_4, & B_6 &\lineqv e_6-e_9,\\
B_7 &\lineqv e_3-e_7, & B_8 &\lineqv e_5-e_8, & F_1 &\lineqv 3h-e_1-\dotsb-e_9,\\
F_2 &\lineqv 3h-e_1-\dotsb-e_9, & \ell &\lineqv h-e_3-e_4-e_7.
\end{aligned}
\end{equation*}

Let $L := h-e_3-e_5-e_6$. Note that the divisor
$B=B_1+B_2+B_3+B_4$ is $2$-divisible because of the relation
\begin{equation}\label{equation:2-divisible}
B=B_1+B_2+B_3+B_4 \lineqv 2(h-e_3-e_5-e_6) = 2L.
\end{equation}

\subsection{A rational surface $Z=E(1) \sharp 7 \overline{\mathbb{CP}^2}$}

In the construction of $Z$, we use only one section $S:= e_4$. We
first blow up at the two nodes of the nodal singular fibers $F_1$
and $F_2$ so that we obtain a blown-up rational elliptic surface
$W=E(1) \sharp 2\overline{\mathbb{CP}^2}$; Figure~\ref{figure:W}.
Let $E_1$ and $E_2$ be the exceptional curves over the nodes of
$F_1$ and $F_2$, respectively. We further blow up at each three
marked points $\bullet$ in Figure~\ref{figure:W}, and we blow up twice at the marked point $\bigodot$ (that is, we first blow-up $\bigodot$ and then again on the intersection point of the section and the exceptional curve; see Figure~\ref{figure:Z}). We then get $Z=E(1) \sharp 7
\overline{\mathbb{CP}^2}$ as in Figure~\ref{figure:Z}. There exist
two linear chains of the $\mathbb{CP}^1$ in $Z$ whose dual graphs
are given by:
\begin{equation*}
C_{8,5}=\uc{-2}-\uc{-3}-\uc{-5} -\uc{-3}, \qquad C_{7,4}=\uc{-2}-\uc{-6}-\uc{-2} -\uc{-3},
\end{equation*}
where $C_{8,5}$ consists of $\ell$, $S$, $F_1$, $E_1$, and $C_{7,4}$ contains $F_2$, $E_2$, $M$.

\begin{figure}[tbh]
\centering
\includegraphics[scale=1]{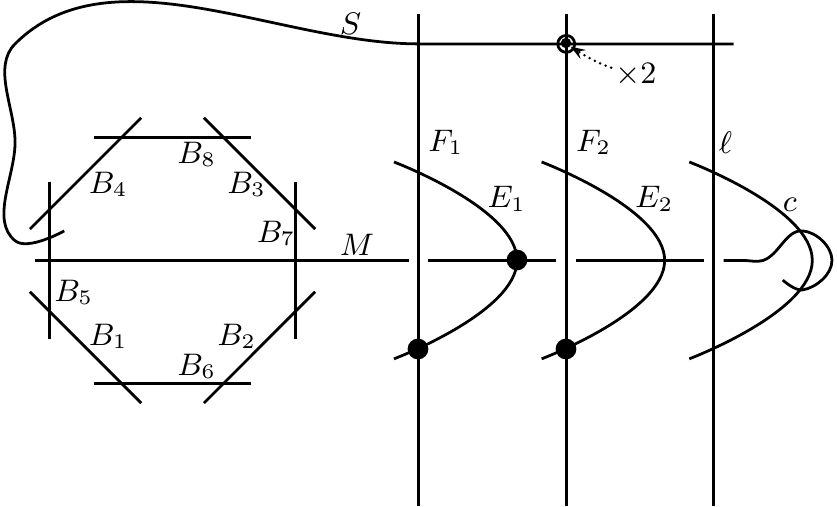}
\caption{A blown-up rational elliptic surface $W=E(1) \sharp 2\overline{\mathbb{CP}^2}$}
\label{figure:W}
\end{figure}

\begin{figure}[tbh]
\centering
\includegraphics[scale=1]{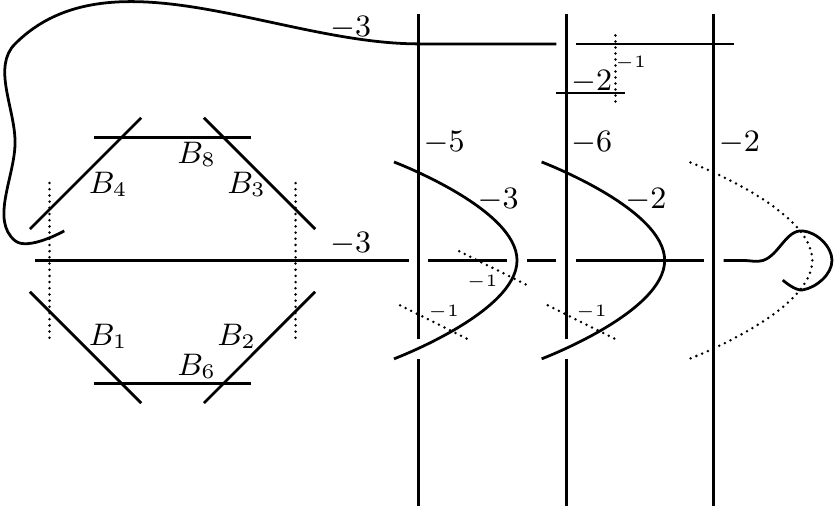}
\caption{A rational surface $Z = E(1) \sharp
7\overline{\mathbb{CP}^2}$} \label{figure:Z}
\end{figure}

\subsection{Numerical Godeaux surfaces}

We construct rational singular surfaces which produce under
$\mathbb{Q}$-Gorenstein smoothings simply connected surfaces of
general type with $p_g=0$ and $K^2=1$.

We first contract the two chains $C_{8,5}$ and $C_{7,4}$ of
$\mathbb{CP}^1$'s from the surface $Z$ so that we have a normal
projective surface $\widetilde{Y}$ with two singularities
$p_1$, $p_2$ of class $T$: $\frac{1}{8^2}(1,8\cdot 5-1),
\frac{1}{7^2}(1,7 \cdot 4-1)$. Denote the contraction morphism by
$\widetilde{\alpha} \colon Z \to \widetilde{Y}$. Let $Y$ be the
surface obtained by contracting the four $(-2)$-curves $B_1,
\dotsc, B_4$ in $\widetilde{Y}$. We denote the contraction
morphism by $\alpha \colon \widetilde{Y} \to Y$. Then $Y$ is also
a normal projective surface with singularities $p_1$, $p_2$ from
$\widetilde{Y}$, and four $A_1$'s (ordinary double points), denoted
by $q_1, \dotsc, q_4$. We finally contract $C_{8,5}$, $C_{7,4}$,
$B_4+B_8+B_3$ and $B_1+B_6+B_2$ in $Z$ to obtain $Y'$. It has the
singularities $\frac{1}{8^2}(1,8\cdot 5-1), \frac{1}{7^2}(1,7
\cdot 4-1)$, and two $A_3=\frac{1}{4}(1,3)$'s. Let $\alpha' \colon Z
\rightarrow Y'$ be the contraction.

In Section~\ref{section:obstruction} we will prove that the
obstruction spaces to local-to-global deformations of the singular
surfaces $\widetilde{Y}$, $Y$ and $Y'$ vanish. That is:

\begin{theorem}\label{theorem:H2(TY)=0}
$H^2(\widetilde{Y}, \sheaf{T_{\widetilde{Y}}})=0$, $H^2(Y,
\sheaf{T_Y})=0$, and $H^2(Y', \sheaf{T_{Y'}})=0$.
\end{theorem}

The singular surface $Y'$ is the stable model of the singular surfaces $\widetilde{Y}$ and $Y$:

\begin{proposition}\label{propsotion:stable-godeaux}
The surface $Y'$ has $K_{Y'}^2=1$, $p_g(Y')=0$, and $K_{Y'}$ is
ample. The space $\Defq(Y')$ is smooth and $8$ dimensional. A
$\mathbb{Q}$-Gorenstein smoothing of $Y'$ is a simply connected canonical
surface of general type with $p_g=0$ and $K^2=1$.
\end{proposition}

\begin{proof}
For a surface $Y'$ with only singularities of type $T$ we have
\begin{equation*}
K_{Y'}^2 = K_{Z}^2 + \sum_{p \in \text{Sing}(Y')} s_p -
\sum_{p \in \text{Sing}(Y')} \mu_p,
\end{equation*}
where $s_p$ is the number of
exceptional curves over $p$ and $\mu_p$ is the Milnor number of
$p$. In our case, $K_{Y'}^2=-7 + 4 + 4 =1$. We have
$p_g(Y')=q(Y')=0$ because of the rationality of the singularities.

We now compute $\alpha'^{\ast}(K_{Y'})$ in a $\mathbb{Q}$-numerically effective way. Let $F$ be the general fiber of the elliptic
fibration in $Z$. Let $E_3$, $E_4$, $E_5$ be the exceptional
curves over $M \cap E_1$, $F_1 \cap E_1$, $F_2 \cap E_2$
respectively. Let $E_6$, $E_7$ be the exceptional curves over $F_2
\cap S$ with $E_7^2=-1$. Then, $K_Z \lineqv -F + \sum_{i=1}^7 E_i +
E_3 + E_4 + E_5 + E_7$. We also have $F \lineqv F_1 + 2 E_1 + 2E_3 +
3 E_4 \lineqv F_2 + 2E_2 + 3 E_5 + E_6 + E_7$. Writing $F \equiv
\frac{1}{2} F + \frac{1}{2} F$ in $K_Z$ and adding the
discrepancies from $p_1$ and $p_2$, we obtain
\begin{equation*}
\begin{split}
\alpha'^{\ast}(K_{Y'})
\equiv &\frac{3}{8} F_1 + \frac{5}{14} F_2 + \frac{5}{8} E_1 +
\frac{5}{7} E_2 + E_3\\
&+\frac{1}{2} E_4 +
\frac{1}{2} E_5 + \frac{13}{14} E_6 + \frac{3}{2} E_7 +
\frac{4}{7} M + \frac{3}{8} \ell + \frac{6}{8} S .
\end{split}
\end{equation*}

We now intersect $\alpha'^{\ast}(K_{Y'})$ with all the curves in its support,
which are not contracted by $\alpha'$, to check that $K_{Y'}$ is
nef. Moreover, if $\Gamma.\alpha'^{\ast}(K_{Y'})=0$ for a curve
$\Gamma$ not contracted by $\alpha'$, then $\Gamma$ is a component
of a fiber in the elliptic fibration which does not intersect any
curve in the support. This is because $F_1$, $E_1$, $E_3$, and
$E_4$ belong to the support, and they are the components of a
fiber. One easily checks that $\Gamma$ does not exist, proving
that $K_{Y'}$ is ample. Therefore any $\mathbb{Q}$-Gorenstein smoothing of
$Y'$ over a (small) disk will produce canonical surfaces; cf.~Koll\'ar-Mori~\cite[p.\thinspace 34]{Kollar-Mori}.

To compute the fundamental group of a $\mathbb{Q}$-Gorenstein smoothing we
use the recipe in Y. Lee-J. Park~\cite{Lee-Park-K^2=2}. We follow the argument as
in Y. Lee-J. Park~\cite[p.\thinspace 493]{Lee-Park-K^2=2}. Consider the normal circles
around $M$ and $E_1$. We can compare them through the transversal
sphere $E_3$. Since the orders of the circles are $49$ and $64$,
which are coprime, we obtain that both end up being trivial.

The smoothness of $\Defq(Y')$ follows from
Theorem~\ref{theorem:H2(TY)=0} and Hacking~\cite[\S3]{Hacking}. To
compute the dimension, we observe that if $\mathcal{Y}'
\rightarrow \Delta$ is a $\mathbb{Q}$-Gorenstein smoothing of
$\mathcal{Y}'_0=Y'$ and $\mathcal{T}_{\mathcal{Y}'|\Delta}$ is the
dual of $\Omega_{\mathcal{Y}'|\Delta}^1$, then
$\mathcal{T}_{\mathcal{Y}'|\Delta}$ restricts to $\mathcal{Y}'_t$
as $\mathcal{T}_{\mathcal{Y}'_t}$ (tangent bundle of
$\mathcal{Y}'_t$) when $t\neq 0$, and
$\mathcal{T}_{\mathcal{Y}'|\Delta}|_{\mathcal{Y}'_0} \subset
\mathcal{T}_{\mathcal{Y}'_0}$ with cokernel supported at the
singular points of $\mathcal{Y}'_0$; cf.~Wahl~\cite{Wahl81}. Then
the flatness of $\mathcal{T}_{\mathcal{Y}'|\Delta}$ and
semicontinuity in cohomology plus the fact that $H^2(Y',
\mathcal{T}_{Y'})=0$ gives
$H^2(\mathcal{Y}'_t,\mathcal{T}_{\mathcal{Y}'_t})=0$ for any $t$.
But then, since $\mathcal{Y}'_t$ is of general type,
Hirzebruch-Riemann-Roch Theorem says
\[H^1(\mathcal{Y}'_t,\mathcal{T}_{\mathcal{Y}'_t})=10
\chi(\mathcal{Y}'_t,\mathcal{O}_{\mathcal{Y}'_t}) - 2
K_{\mathcal{Y}'_t}^2=10-2=8.\]
This proves the claim.
\end{proof}

\begin{corollary}\label{corollary:Godeaux}
There is a two dimensional family of simply connected canonical
numerical Godeaux surfaces with two $A_3$ singularities.
\label{Corollary:six}
\end{corollary}

\begin{proof}
We consider the sequence $$ 0 \rightarrow H^1(Y',
\mathcal{T}_{Y'}) \rightarrow T_{\text{QG},Y'}^1 \rightarrow
H^0(Y', \mathcal{T}_{\text{QG},Y'}^1) \rightarrow 0 $$ at the end
of Hacking~\cite[\S3]{Hacking}. We just proved that
$T_{\text{QG},Y'}^1$ is $8$ dimensional, and we know that $H^0(Y',
\mathcal{T}_{\text{QG},Y'}^1)$ is $8$ dimensional, since each
$A_3$ gives $3$ dimensions and each $p_i$ gives $1$ dimension.
Therefore $H^1(Y', \mathcal{T}_{Y'})=0$. To produce the claimed
family we need to smooth up at the same time $p_1$ and $p_2$.
\end{proof}

A simply connected numerical Godeaux surfaces with a $2$-divisible
divisor consisting of four disjoint $(-2)$-curves is obtained from
a $\mathbb{Q}$-Gorenstein smoothing of the singular surface
$\widetilde{Y}$:

\begin{theorem}\label{theorem:Q-smoothing-of-Y}
\begin{enumerate}[(a)]
\item There is a $\mathbb{Q}$-Gorenstein smoothing
$\widetilde{\mathcal{Y}} \to \Delta$ over a disk $\Delta$ with
central fiber $\widetilde{\mathcal{Y}}_0=\widetilde{Y}$ and an
effective divisor $\mathcal{B} \subset \widetilde{\mathcal{Y}}$
such that the restriction to a fiber $\widetilde{Y}_t$ over $t \in \Delta$
\[B_t := \mathcal{B} \cap \widetilde{Y}_t=B_{1,t}+B_{2,t}+B_{3,t}+B_{4,t}\]
is $2$-divisible in $\widetilde{Y}_t$ consisting of four
disjoint $(-2)$-curves and $B_0=B$.

\item There is a $\mathbb{Q}$-Gorenstein deformation $\mathcal{Y}
\to \Delta$ of $Y$ with central fiber $\mathcal{Y}_0=Y$ such that a fiber $Y_t$ over $t \neq 0$ has four
ordinary double points as its only singularities and the minimal
resolution of $Y_t$ is the corresponding fiber
$\widetilde{Y}_t$ of $\widetilde{\mathcal{Y}} \to
\Delta$.
\end{enumerate}
\end{theorem}

\begin{proof}
We apply a similar method in Y. Lee-J.
Park~\cite{Lee-Park-Horikawa}. Since any local deformations of the
singularities of $Y$ can be globalized by
Theorem~\ref{theorem:H2(TY)=0}, there are $\mathbb{Q}$-Gorenstein
deformations of $Y$ over a disk $\Delta$ which keep all four
ordinary double points and smooth up $p_1$ and $p_2$. Let
$\mathcal{Y} \to \Delta$ be such deformation, with $Y$ as its central fiber, and $\mathcal{Y}_t$ ($t \neq 0$) a normal
projective surface with four $A_1$s as its only singularities. We
resolve simultaneously these four singularities in each fiber
$\mathcal{Y}_t$. We then get a family $\widetilde{\mathcal{Y}} \to
\Delta$ that is a $\mathbb{Q}$-Gorenstein smoothing of the central
fiber $\widetilde{Y}$, which shows
that a $\mathbb{Q}$-Gorenstein smoothing of $Y$ can be lifted to a
$\mathbb{Q}$-Gorenstein smoothing of the pair $(Y,B)$, i.e. the
$2$-divisible divisor $B$ on $\widetilde{Y}$ is extended to an
effective divisor $\mathcal{B} \subset \widetilde{\mathcal{Y}}$.

We finally show that the effective divisor $B_t$ is 2-divisible in
$\widetilde{Y}_t$ for $t \neq 0$. According to
Manetti~\cite[Lemma~2]{Menetti}, the natural restriction map $r_t:
\pic(\widetilde{\mathcal{Y}}) \to \pic(\widetilde{Y}_t)$
is injective for every $t \in \Delta$ and bijective for $0 \in
\Delta$. Here we are using that
$p_g(\widetilde{Y})=q(\widetilde{Y})=0$.
Since the divisor $B$ is nonsingular, $\mathcal{B}_t$ is also
nonsingular. Since $B \lineqv 2L$ in \eqref{equation:2-divisible},
it follows that $B_t \lineqv 2L_t$, where $L$ is extended to a
line bundle $\mathcal{L} \subset \widetilde{\mathcal{Y}}$ and
$L_t$ is the corresponding restriction.
\end{proof}

\section{Numerical Campedelli surfaces with an involution}
\label{section:Campedelli}

The main purpose of this section is to construct simply connected
numerical Campedelli surfaces with an involution. Along the way,
we will introduce a rational normal projective surface $X'$ with
$6$ singularities (two $A_1$, two $\frac{1}{8^2}(1,8\cdot 5-1)$,
and two $\frac{1}{7^2}(1,7 \cdot 4-1)$) and $K_{X'}$ ample. A
certain four dimensional $\mathbb{Q}$-Gorenstein deformation of $X'$ will
produce numerical Campedelli surfaces with an involution.

Recall that the rational surface $Z$ has a $2$-divisible divisor
$B=B_1+B_2+B_3+B_4$; cf.~\eqref{equation:2-divisible}. Let $V$ be
the double cover of $Z$ branched along the divisor $B$, where the
double cover is given by the data $B \lineqv 2L$, $L = h-e_3-e_5-e_6$. We denote the double covering by $\psi: V \to Z$. The surface $V$ has two $C_{8,5}$'s and two $C_{7,4}$'s.

On the other hand the surface $V$ can be obtained from a certain rational elliptic surface by blowing-ups, as we now explain. The morphism $\psi \colon V \to Z$ blows down to a double cover $\psi' \colon V' \to E(1)$ branched along $B$. The ramification divisor $\psi'^{-1}(B)$ consists of four disjoint $(-1)$-curves $R_1, \dotsc, R_4$. We blow down them from $V'$ to obtain a surface $E(1)'$; cf.~Figure~\ref{figure:E'}. In Figure~\ref{figure:E'} the pull-back of the $B_i$ are the $B'_i$, of the curve $\ell$ is $\ell_1+\ell_2$, of the section $S$ is $S_1 + S_2$, and of the double section $M$ is $M_1+ M_2$. Each $I_2$ in $E(1)'$ is the pull-back of each $I_1$ in $E(1)$.

\begin{figure}[tbh]
\centering
\includegraphics[scale=1]{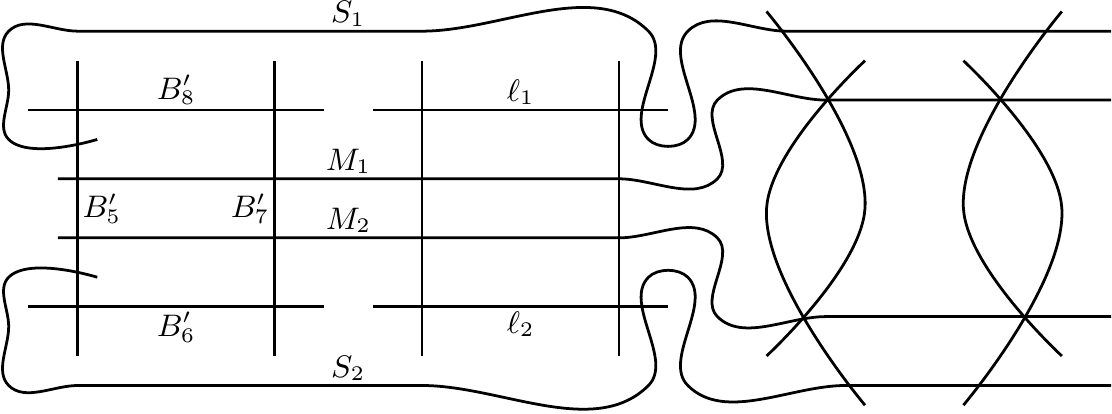}
\caption{Relevant curves in the elliptic rational surface $E(1)'$}
\label{figure:E'}
\end{figure}

Note that the surface $E(1)'$ has an elliptic fibration structure with two $I_4$-singular fibers and two $I_2$-singular fibers. In fact, the surface $E(1)'$ can
be obtained from the pencil of cubics in $\mathbb{CP}^2$
\[\{\lambda x(y-z)(x-2z) + \mu y(x-z)(y-2z) \mid [\lambda:\mu] \in
\mathbb{CP}^1 \}\]
where the $I_4$-singular fibers come from $x(y-z)(x-2z)=0$
and $y(x-z)(y-2z)=0$, and the $I_2$-singular fibers come from $(x+y-2z)(xy-yz-xz)=0$
and $(x-y)(xy-xz+2z^2-yz)=0$. The two double sections $M_1$ and $M_2$
are defined by the lines $x=(1+\sqrt{-1})z$ and $x=(1-\sqrt{-1})z$. In summary:

\begin{proposition}\label{proposition:CampExplicit}
The surface $V'$ is the blow-up at four nodes of one
$I_4$-singular fiber of the rational elliptic fibration $E(1)'
\rightarrow \mathbb{CP}^1$. Hence the surface $V$ can be obtained
from $V'$ by blowing-up in the obvious way.
\end{proposition}

Let $\widetilde{X}$ be the double cover of the singular
surface $\widetilde{Y}$ branched along the divisor $B$. Note that
the surface $\widetilde{X}$ is a normal projective surface with
four singularities of class $T$ whose resolution graphs consist of
two $C_{8,5}$'s and two $C_{7,4}$'s. The ramification divisor in
$\widetilde{X}$ consists of the four disjoint $(-1)$-curves $R_1,
\dotsc, R_4$. Let $\widetilde{\phi} : \widetilde{X} \to
\widetilde{Y}$ be the double covering. On the other hand the
surface $\widetilde{X}$ can be obtained from the rational surface
$V$ by contracting the two $C_{8,5}$'s and two $C_{7,4}$'s. Let
$\widetilde{\beta}: V \to \widetilde{X}$ be the contraction
morphism.

Let $X$ be the surface obtained by blowing down the four
$(-1)$-curves $R_1, \dotsc, R_4$ from $\widetilde{X}$. We denote
the blowing-down morphism by $\beta: \widetilde{X} \to X$. Then
there is a double covering $\phi: X \to Y$ branched along the four
ordinary double points $q_1, \dotsc, q_4$. Finally, let $X'$ be
the contraction of the $(-2)$-curves $B'_8$ and $B'_6$ in $X$. Let
$\beta': V \to X'$ be the contraction. We then get a double covering $X' \to Y'$. To sum up, we have the
following commutative diagram:
\begin{equation*}
\xymatrix{
E(1)'  & V' \ar[l] \ar[d]^{\psi'} & V \ar[l] \ar[d]^{\psi} \ar[r]^{\widetilde{\beta}} & \widetilde{X} \ar[d]^{\widetilde{\phi}} \ar[r]^{\beta} & X \ar[d]^{\phi} \ar[r] & X' \ar[d]\\
 & E(1) & \ar[l] Z \ar[r]^{\widetilde{\alpha}} & \widetilde{Y} \ar[r]^{\alpha} & Y \ar[r] & Y'}
\end{equation*}

We will show in Section~\ref{section:obstruction} the
obstruction spaces to local-to-global deformations of the singular
surfaces $\widetilde{X}$, $X$ and $X'$ vanish:

\begin{theorem}\label{theorem:H2(TX)=0}
$H^2(\widetilde{X},
\sheaf{T_{\widetilde{X}}})=0$, $H^2(X, \sheaf{T_X})=0$, and
$H^2(X',\sheaf{T_{X'}})=0$.
\end{theorem}

The singular surface $X'$ is the stable model of $\widetilde{X}$ and $X$:

\begin{proposition}\label{proposition:stable-campedelli}
The surface $X'$ has $K_{X'}^2=2$, $p_g(X')=0$, and $K_{X'}$
ample. The space $\Defq(X')$ is smooth and $6$ dimensional. A
$\mathbb{Q}$-Gorenstein smoothing of $X'$ is a simply connected canonical
surface of general type with $p_g=0$ and $K^2=2$.
\end{proposition}

\begin{proof}
The proof goes as the one for $Y'$ in Proposition~\ref{propsotion:stable-godeaux}, using the explicit model we have for $V$ by blowing-up $E(1)'$ in Proposition~\ref{proposition:CampExplicit}.
One can check that an intersection computation as in Proposition~\ref{propsotion:stable-godeaux} verifies ampleness for $K_{X'}$.
\end{proof}

The proof of the next main result follows easily from Theorem~\ref{theorem:Q-smoothing-of-Y}.

\begin{theorem}\label{theorem:campedelli}
There exist $\mathbb{Q}$-Gorenstein smoothings
$\widetilde{\mathcal{X}} \to \Delta$ of $\widetilde{X}$ and
$\mathcal{X} \to \Delta$ of $X$ that are compatible with the
$\mathbb{Q}$-Gorenstein deformations of $\widetilde{\mathcal{Y}}
\to \Delta$ of $\widetilde{Y}$ and $\mathcal{Y} \to \Delta$ of $Y$
in Theorem~\ref{theorem:Q-smoothing-of-Y}, respectively;
that is, the double coverings $\widetilde{\phi}: \widetilde{X} \to
\widetilde{Y}$ and $\phi: X \to Y$ extend to the double coverings
$\widetilde{\phi}_t: \widetilde{X}_t \to \widetilde{Y}_t$ and
$\phi_t: X_t \to Y_t$ between the fibers of the
$\mathbb{Q}$-Gorenstein deformations.
\end{theorem}

\begin{remark}
By Theorem~\ref{theorem:H2(TX)=0}, the obstruction $H^2(X, \sheaf{T_X})$ to local-to-global deformations of the singular surface $X$ vanishes. The point of the above theorem is that there is a $\mathbb{Q}$-Gorenstein smoothing of the cover $X$ that is compatible with the $\mathbb{Q}$-Gorenstein deformation of the base $Y$.
\end{remark}

\begin{corollary}\label{corollary:Campedelli-with-involution}
A general fiber $X_t$ of the $\mathbb{Q}$-Gorenstein smoothing $\mathcal{X} \to \Delta$ of $X$ is a simply connected numerical Campedelli surface with an involution $\sigma_t$ such that the minimal resolution of the quotient $Y_t=X_t/\sigma_t$ is a simply connected numerical Godeaux surface.
\end{corollary}

\subsection{The fundamental group of the quotient by an involution}

Let $X$ be a minimal complex surface of general type with $p_g=0$
and $K^2=2$. Suppose that the group $\mathbb{Z}/2\mathbb{Z}$ acts
on $X$ with just $4$ fixed points. Let
$Y=X/(\mathbb{Z}/2\mathbb{Z})$ be the quotient and let $S \to Y$
be the minimal resolution of $Y$.
Barlow~\cite[Proposition~1.3]{Barlow-Invent} proved that if
$\pi_1^{\text{alg}}(S)=1$ then
$\abs{\pi_1^{\text{alg}}(X)}=1,3,5,7,9$. Conversely:

\begin{proposition}\label{proposition:simply-connected=>simply-connected}
If $X$ is simply connected, then so is $S$.
\end{proposition}

\begin{proof}
Let $f: X \to Y$ be the quotient map. Then $f$ is a double
covering which is branched along the four ordinary double points of $Y$. Let $Y_0 \subset Y$ be the complement of the four branch points (i.e., the four $A_1$-singularities) of $Y$ and let $X_0=f^{-1}(Y_0)$, that is, $X_0 \subset X$ is the complement of the four fixed points of the involution $\sigma$. Then we get an \'etale double covering $f|_{X_0}: X_0 \to Y_0$. Since $\pi_1(X_0)=\pi_1(X)=1$, we have $\pi_1(Y_0) = \mathbb{Z}/2\mathbb{Z}$. Note that the boundary $\partial U$ of an arbitrary small neighborhood $U$ of one of the four nodes of $Y$ is a Lens space $L(2,1)$. Let $[\gamma]$ be a generator of $\pi_1(\partial U) \cong \mathbb{Z}/2\mathbb{Z}$ represented by a loop $\gamma$ contained in $\partial U$. Since the lifting of $\gamma$ by the covering $f|_{X_0}: X_0 \to Y_0$ is not a closed path and $\pi_1(Y_0)=\mathbb{Z}/2\mathbb{Z}$, $\pi_1(Y_0)$ is generated by $[\gamma]$. Then it follows by van Kampen theorem that $\pi(Y)$ is trivial. Hence $\pi_1(S)$ is trivial because $S$ is obtained from $Y$ by resolving only four $A_1$-singularities.

%
\end{proof}

\section{The obstruction spaces to local-to-global deformations}
\label{section:obstruction}

In this section we prove Theorem~\ref{theorem:H2(TY)=0} which says
that the obstruction spaces to local-to-global deformations of the
singular surfaces $\widetilde{Y}$, $Y$, and $Y'$ vanish. That is,
we will prove that $H^2(\widetilde{Y},
\sheaf{T_{\widetilde{Y}}})=H^2(Y, \sheaf{T_Y})=H^2(Y',
\sheaf{T_{Y'}})=0$. At the end, we also prove the analogues,
Theorem~\ref{theorem:H2(TX)=0}, for $\widetilde{X}$, $X$, and
$X'$.

At first the vanishing of the obstruction
spaces of a singular surface can be proved by the vanishing of the second cohomologies
of a certain logarithmic tangent sheaf on the minimal resolution
of the singular surface:

\begin{proposition}[{Y.~Lee-J.~Park~\cite[Theorem~2]{Lee-Park-K^2=2}}]\label{proposition:Lee-Park-Log}
If $\pi: T \to S$ be the minimal resolution of a
normal projective surface $S$ with only quotient singularities,
and $D$ is the reduced exceptional divisor of the resolution
$\pi$, then $h^2(S, \sheaf{T_S}) = h^2(T,
\sheaf{T_{T}}(-\log{D}))$.
\end{proposition}

\begin{proposition}[{Flenner-Zaidenberg~\cite[Lemma~1.5]{Flenner-Zaidenberg}}]\label{proposition:Flenner-Zaidenberg}
Let $T$ be a nonsingular surface and let $D$ be a simple normal
crossing divisor in $T$. Let $f : T' \to T$ be the
blow-up of $T$ at a point $p$ of $D$. Let
$D'=f^{*}(D)_{\text{red}}$. Then $h^2(T',
\sheaf{T_{T'}}(-\log{D'}))=h^2(T,
\sheaf{T_{T}}(-\log{D}))$.
\end{proposition}

We can add or remove disjoint $(-1)$-curves.

\begin{proposition} \label{add-delete}
Let $T$ be a nonsingular surface and let $D$ be a simple normal
crossing divisor in $T$. Let $E$ be a $(-1)$-curve in $T$ such
that $D+E$ is again simple normal crossing. Then $h^2(T,
\sheaf{T_{T}}(-\log{(D+E)}))= h^2(T, \sheaf{T_{T}}(-\log{D}))$.
\end{proposition}

We can also add or remove disjoint exceptional divisors of
rational double points. The following theorem may give a new
general way to prove unobstructedness for deformations of
surfaces.

\begin{theorem}\label{theorem:Burns-Wahl}
Let $S$ be a normal projective surface with only rational double
points $q_1, \ldots,q_n$ as singularities. Let $\pi \colon
\widetilde{S} \to S$ be the minimal resolution of $S$ with
exceptional reduced divisor $M=\sum_{i=1}^n \pi^{-1}(q_i)$. Let
$C$ be a simple normal crossing divisor such that $C\cap
M=\varnothing$. Then $h^2(\widetilde{S},
\sheaf{T_{\widetilde{S}}}(-\log(C+M)))=h^2(\widetilde{S},
\sheaf{T_{\widetilde{S}}}(-\log{C}))$.
\end{theorem}

\begin{proof}
Let $M=\sum_i M_i$ and $C=\sum_i C_i$ be the prime decompositions
of $M$ and $C$. We have three short exact sequences:

\begin{align*}\allowdisplaybreaks
&0 \to \sheaf{T_{\widetilde{S}}}(-\log{M}) \to \sheaf{T_{\widetilde{S}}} \to \oplus \sheaf{N_{M_i/{\widetilde{S}}}} \to 0,\\
&0 \to \sheaf{T_{\widetilde{S}}}(-\log(C+M)) \to \sheaf{T_{\widetilde{S}}}(-\log{M}) \to \oplus \sheaf{N_{C_i/{\widetilde{S}}}} \to 0,\\
&0 \to \sheaf{T_{\widetilde{S}}}(-\log(C+M)) \to
\sheaf{T_{\widetilde{S}}}(-\log{C}) \to \oplus
\sheaf{N_{M_i/{\widetilde{S}}}} \to 0.
\end{align*}
We then have the following commutative diagram of cohomologies:
\begin{equation*}
\small{ \xymatrix{
  & 0 \ar[d] & 0 \ar[d] &  &\\
0 \ar[r] & H^1(\sheaf{T_{\widetilde{S}}}(-\log(C+M))) \ar[r] \ar[d] & H^1(\sheaf{T_{\widetilde{S}}}(-\log{C}))
\ar[r]^{\phi} \ar[d]^{\xi} & \oplus H^1(\sheaf{N_{M_i/{\widetilde{S}}}}) \ar@{=}[d]& \\
0 \ar[r] & H^1(\sheaf{T_{\widetilde{S}}}(-\log{M})) \ar[r] \ar[d] & H^1(\sheaf{T_{\widetilde{S}}}) \ar[r]^{\psi}
\ar[d]^{\zeta}& \oplus H^1(\sheaf{N_{M_i/{\widetilde{S}}}}) \ar[r] & 0\\
 & \oplus H^1(\sheaf{N_{C_i/{\widetilde{S}}}}) \ar@{=}[r] & \oplus H^1(\sheaf{N_{C_i/{\widetilde{S}}}})& &
}}
\end{equation*}

Here all horizontal and vertical sequences are exact. Especially
the second row is a short exact sequence, which we explain now briefly: It is shown
in Burns-Wahl~\cite[pp.\thinspace 70--72]{Burns-Wahl} (see also Wahl~\cite[\S6]{Wahl75})
that the composition
\begin{equation*}
H^1_M(\sheaf{T_{\widetilde{S}}}) \to H^1(
\sheaf{T_{\widetilde{S}}}) \to \oplus H^1(
\sheaf{N_{M_i/\widetilde{S}}})
\end{equation*}
is an isomorphism because the $q_i$'s are \emph{rational
double points}; hence, one has a direct sum decomposition
\begin{equation}\label{equation:H^1(T_Z)=H^1(T_Z(-log(M)))}
H^1(\sheaf{T_{{\widetilde{S}}}}) =
H^1(\sheaf{T_{\widetilde{S}}}(-\log{M})) \oplus H^1_M(
\sheaf{T_{\widetilde{S}}})
\end{equation}
and an isomorphism $H^2(\sheaf{T_{\widetilde{S}}}) \cong
H^2(\sheaf{T_{\widetilde{S}}}(-\log{M}))$. Therefore the second row is exact.

In order to prove the assertion, it is enough to show that
\[\phi: H^1(\sheaf{T_{\widetilde{S}}}(-\log{C}))\to \oplus H^1(\sheaf{N_{M_i/{\widetilde{S}}}})\]
is surjective. Let $\alpha \in \oplus
H^1(\sheaf{N_{M_i/{\widetilde{S}}}})$. Since $\psi$ is surjective,
we have $\psi(\beta)=\alpha$ for some $\beta \in
H^1(\sheaf{T_{\widetilde{S}}})$. By
\eqref{equation:H^1(T_Z)=H^1(T_Z(-log(M)))} we have
\[\beta = \gamma + \alpha'\]
for some $\gamma \in H^1(\sheaf{T_{\widetilde{S}}}(-\log{M}))$ and
$\alpha' \in H^1_M(\sheaf{T_{\widetilde{S}}})$ such that $\alpha'$
is mapped to $\alpha$ under the composition
$H^1_M(\sheaf{T_{\widetilde{S}}}) \to
H^1(\sheaf{T_{\widetilde{S}}}) \to \oplus
H^1(\sheaf{N_{M_i/{\widetilde{S}}}})$. Since $\alpha'$ is
supported on $M$ and $C \cap M = \varnothing$, its image
$\zeta(\alpha')$ under $\zeta: H^1(\sheaf{T_{\widetilde{S}}}) \to
\oplus H^1(\sheaf{N_{C_i/{\widetilde{S}}}})$ vanishes. Therefore
$\zeta(\beta-\gamma)=\zeta(\alpha')=0$, and so
\[\beta-\gamma=\xi(\delta)\]
for some $\delta \in H^1(\sheaf{T_{\widetilde{S}}}(-\log{C}))$;
hence, $\phi(\delta)=\psi(\xi(\delta))=\alpha$, which shows that
$\phi$ is surjective.
\end{proof}

\begin{proposition}[cf.~{Esnault-Viehweg~\cite[2.3]{Esnault-Viehweg}}]\label{proposition:E-V-sequence}
Let $E = \sum_{i=1}^{n}{C_i}$ be a simple normal crossing divisor on a smooth surface $T$. Then one has the following exact sequences:

    \begin{enumerate}[(a)]
    \item $\displaystyle 0 \to \Omega_T^1 \to \Omega_T^1 (\log{E}) \to \bigoplus_{i=1}^{n} \sheaf{O_{C_i}} \to 0$.

    \item $\displaystyle 0 \to \Omega_T^1(\log{E}) \to \Omega_T^1(\log(E-C_1))(C_1) \to \Omega_{C_1}^1(E|_{C_1}) \to 0$.
    \end{enumerate}
\end{proposition}

\begin{proof}[Proof of Theorem~\ref{theorem:H2(TY)=0}]
We first claim that
\[H^2(W, \sheaf{T_W}(-\log(F_1+F_2)))=0.\]
By duality, we have to show that
\begin{equation*}
h^2(W, \sheaf{T_W}(-\log(F_1+F_2)))=h^0(W, \Omega_W^1(\log(F_1+F_2))(K_W))=0.
\end{equation*}
Since $K_W \lineqv -F_2+E_1-E_2$, it follows by Proposition~\ref{proposition:E-V-sequence} that
\begin{equation*}\allowdisplaybreaks
\begin{split}
H^0(W, \Omega_W^1(\log(F_1+F_2))(K_W)) &= H^0(W, \Omega_W^1(\log{F_1})(K_W+F_2))\\
&= H^0(W, \Omega_W^1(\log{F_1})(E_1-E_2))\\
&\subset H^0(W, \Omega_W^1(\log{F_1})(E_1)) \\
&=H^0(W, \Omega_W^1(\log(F_1+E_1))).
\end{split}
\end{equation*}
Hence it suffices to show that $H^0(W,
\Omega_W^1(\log(F_1+E_1)))=0$. On the other hand, we obtain a long
exact sequence from Proposition~\ref{proposition:E-V-sequence}:
\[H^0(W, \Omega_W) \to H^0(W, \Omega_W^1(\log(F_1+E_1))) \to H^0(F_1, \sheaf{O_{F_1}}) \oplus H^0(E_1, \sheaf{O_{E_1}}) \xrightarrow{\delta} H^1(W, \Omega_{W})\]
Since $H^0(W, \Omega_W)=0$, it is enough to show that the
connecting homomorphism $\delta$ is injective. Note the map
$\delta$ is the first Chern class map. But $F_1$ and $E_1$ are
linearly independent in the Picard group of $W$; hence, the map
$\delta$ is injective. Therefore the claim follows.

Let $B'=B_4+B_8+B_3+B_1+B_6+B_2$. By Theorem~\ref{theorem:Burns-Wahl} we have
\[h^2(W,\sheaf{T_W}(-\log{(F_1+F_2)})) = h^2(W,\sheaf{T_W}(-\log{(F_1+F_2+B'+M+\ell)})).\]
We use Propositions~\ref{add-delete} and
\ref{proposition:Flenner-Zaidenberg} to obtain
\begin{multline*}
h^2(W,\sheaf{T_W}(-\log{(F_1+F_2+B'+M+\ell)}))=\\
h^2(Z,\sheaf{T_Z}(-\log{(F_1+F_2+B'+M+\ell+S+E')})),
\end{multline*}
where $E'=E_1+E_2+E_6$. In this way, it follows by the above claim that
\[h^2(Z,\sheaf{T_Z}(-\log{(F_1+F_2+B'+M+l+S+E')}))=h^2(W,\sheaf{T_W}(-\log{(F_1+F_2)}))=0.\]
Then, by Proposition~\ref{proposition:Lee-Park-Log}, we have $H^2(Y,
\sheaf{T_{Y'}})=0$. Notice we can modify $B'$ to obtain vanishing
for $H^2(\widetilde{Y}, \sheaf{T_{\widetilde{Y}}})$ and $H^2(Y,
\sheaf{T_Y})$ as well.
\end{proof}

We now prove that $H^2(\widetilde{X},
\sheaf{T_{\widetilde{X}}})=H^2(X, \sheaf{T_X})=H^2(X',
\sheaf{T_{X'}})=0$.

\begin{proof}[Proof of Theorem~\ref{theorem:H2(TX)=0}]
We will use our explicit model of $X'$ in
Proposition~\ref{proposition:CampExplicit}. The proof goes along
the same lines as the proof of the above
Theorem~\ref{theorem:H2(TY)=0}. We may only need to mention that
we start with the elliptic fibration $E(1)'$, and $2$ $I_2$ fibers
(instead of $2$ $I_1$'s).
\end{proof}

\section{The invariant part of the deformation space} \label{section:InvDeformation-space}

The involution of a general fiber $X_t$ induced by the double
covering $\phi_t: X_t \to Y_t$ extends to a
$\mathbb{Z}/2\mathbb{Z}$-action on the deformation space of $X_t$.
We will count the dimension of the subspace of $H^1(X_t,
\sheaf{T_{X_t}})$ which is fixed by the
$\mathbb{Z}/2\mathbb{Z}$-action; Theorem~\ref{theorem:invariant-part}.

Let $\alpha_t: \widetilde{Y}_t \to Y_t$ be the minimal resolution and
let $\beta_t: \widetilde{X}_t \to X_t$ be the blowing-up at the
four ramification points. We then have the following commutative
diagram where the vertical morphisms are double covers:
\begin{equation*}
\xymatrix{
\widetilde{X}_t \ar[r]^{\beta_t} \ar[d]_{\widetilde{\phi}_t} & X_t \ar[d]^{\phi_t} \\
\widetilde{Y}_t \ar[r]^{\alpha_t} & Y_t
}
\end{equation*}
Recall that the branch divisor $B_t$ of the double covering
$\widetilde{\phi}_t$ consists of four disjoint $(-2)$-curves
$B_{1,t}, \dotsc, B_{4,t}$ and the corresponding ramification
divisor $R_t$ consists of four disjoint $(-1)$-curves
$R_{1,t}, \dotsc, R_{4,t}$. As before the involution of
$\widetilde{X}_t$ induced by the double covering
$\widetilde{\phi}_t: \widetilde{X}_t \to \widetilde{Y}_t$ extends
to a $\mathbb{Z}/2\mathbb{Z}$-action on the deformation space of
$\widetilde{X}_t$.

\begin{lemma}\label{lemma:invariant}
$\dim H^1(\widetilde{X}_t, \sheaf{T_{\widetilde{X}_t}})^{\mathbb{Z}/2\mathbb{Z}}=4$.
\end{lemma}

\begin{proof}
By Pardini~\cite[Lemma~4.2]{Pardini}, the invariant part of
$(\widetilde{\phi}_t)_{\ast}{\sheaf{T_{\widetilde{X}_t}}}$ under
the $\mathbb{Z}/2\mathbb{Z}$-action is
$\sheaf{T_{\widetilde{Y}_t}}(-\log{B_t})$. Therefore we have
\begin{equation*}
H^1(\widetilde{X}_t,
\sheaf{T_{\widetilde{X}_t}})^{\mathbb{Z}/2\mathbb{Z}} \cong
H^1(\widetilde{Y}_t, \sheaf{T_{\widetilde{Y}_t}}(-\log{B_t}))
\cong H^1(\widetilde{Y}_t,
\Omega_{\widetilde{Y}_t}^1(\log{B_t})(K_{\widetilde{Y}_t})).
\end{equation*}
We know that ${\widetilde{\phi_t}}_*
(\Omega_{\widetilde{X}_t}^1(K_{\widetilde{X}_t})) =
\Omega_{\widetilde{Y}_t}^1(\log{B_t})(K_{\widetilde{Y}_t}) \oplus
\Omega_{\widetilde{Y}_t}^1(K_{\widetilde{Y}_t}+L_t)$, where $2 L_t
\sim B_t$ defines the double cover $\widetilde{\phi_t}$. Therefore
$ H^2(\sheaf{T_{\widetilde{X}_t}})
=H^0(\Omega_{\widetilde{X}_t}^1(K_{\widetilde{X}_t}))
=H^0({\widetilde{\phi_t}}_*(\Omega_{\widetilde{X}_t}^1(K_{\widetilde{X}_t})))=0$
by Theorem~\ref{proposition:stable-campedelli}. Then we have $
H^0(\Omega_{\widetilde{Y}_t}^1(\log{B_t})(K_{\widetilde{Y}_t}))=0$.
Hence, by Proposition~\ref{proposition:E-V-sequence}, there is a
short exact sequence
\begin{equation*}
0 \to H^0(B_t, \sheaf{O_{B_t}}(K_{\widetilde{Y}_t})) \to
H^1(\widetilde{Y}_t,
\Omega_{\widetilde{Y}_t}^1(K_{\widetilde{Y}_t})) \to
H^1(\widetilde{Y}_t,
\Omega_{\widetilde{Y}_t}^1(\log{B_t})(K_{\widetilde{Y}_t})) \to
0.
\end{equation*}
Since $B_t$ consists of four disjoint $(-2)$-curves, we have
$h^0(B_t, \sheaf{O_{B_t}}(K_{\widetilde{Y}_t}))=4$. By
Proposition~\ref{propsotion:stable-godeaux}, we know that
$h^1(\widetilde{Y}_t,
\Omega_{\widetilde{Y}_t}^1(K_{\widetilde{Y}_t}))=8$. Therefore $
h^1(\widetilde{Y}_t,
\Omega_{\widetilde{Y}_t}^1(\log{B_t})(K_{\widetilde{Y}_t}))=4$.

\end{proof}

\begin{theorem}\label{theorem:invariant-part}
The subspace of the deformation space of $X_t$ invariant under the
$\mathbb{Z}/2\mathbb{Z}$-action is four dimensional.
\end{theorem}

\begin{proof}
We apply a similar strategy in Werner~\cite[\S4]{Werner}. We have the exact sequence
\begin{equation*}
0 \to \sheaf{T_{\widetilde{X}_t}}(-\log{R}) \to \sheaf{T_{\widetilde{X}_t}} \to \oplus \sheaf{N_{R_{i,t}, \widetilde{X}}} \to 0.
\end{equation*}
Since each $R_{i,t}$ is a $(-1)$-curve, we have $
H^1(\widetilde{X}_t, \sheaf{T_{\widetilde{X}_t}}) =
H^1(\widetilde{X}_t, \sheaf{T_{\widetilde{X}_t}}(-\log{R}))$. On
the other hand, it follows by
Catanese~\cite[Lemma~9.22]{Catanese-book} that
\begin{equation}\label{equation:Catanese}
H^1(\widetilde{X}_t, \sheaf{T_{\widetilde{X}_t}}(-\log{R})) = H^1(X_t, \sheaf{T_{X_t}} \otimes \sheaf{I})
\end{equation}
where $\sheaf{I}$ is the ideal sheaf of the four points in $X_t$
obtained by contacting the exceptional divisors $R_{1,t}, \dotsc,
R_{4,t}$.

Let $P$ be the set of these four points. From the ideal sequence,
we have
\begin{equation*}
0 \to H^0(P, \sheaf{T_{X_t}} \otimes \sheaf{O_P}) \to H^1(X_t,
\sheaf{T_{X_t}} \otimes \sheaf{I}) \to H^1(X_t, \sheaf{T_{X_t}})
\to 0.
\end{equation*}
Therefore the invariant parts of each space satisfies:
\begin{equation*}
0 \to H^0(P, \sheaf{T_{X_t}} \otimes
\sheaf{O_P})^{\mathbb{Z}/2\mathbb{Z}} \to H^1(X_t, \sheaf{T_{X_t}}
\otimes \sheaf{I})^{\mathbb{Z}/2\mathbb{Z}} \to H^1(X_t,
\sheaf{T_{X_t}})^{\mathbb{Z}/2\mathbb{Z}} \to 0.
\end{equation*}
According to Werner~\cite[p.\thinspace1523]{Werner}, we have $H^0(P,
\sheaf{T_{X_t}} \otimes \sheaf{O_P})^{\mathbb{Z}/2\mathbb{Z}}=0$.
Therefore it follows by \eqref{equation:Catanese} and
Lemma~\ref{lemma:invariant} that
$$\dim H^1(X_t, \sheaf{T_{X_t}})^{\mathbb{Z}/2\mathbb{Z}} = \dim H^1(X_t, \sheaf{T_{X_t}} \otimes \sheaf{I})^{\mathbb{Z}/2\mathbb{Z}}
=\dim
H^1(\widetilde{X}_t,\sheaf{T_{\widetilde{X}_t}}(-\log{R}))^{\mathbb{Z}/2\mathbb{Z}}=4.$$

\end{proof}


\section{Another example}\label{section:another-example}

We briefly describe another rational surface $Z$ which makes it
possible to construct simply connected numerical Campedelli
surfaces with an involution as before. The associated Godeaux
surfaces come from a rational surface $Y'$ with $K_{Y'}$ ample having three $A_1$-singularities, one $A_3$-singularity, and only one
singularity of class $T$.

\subsection{A rational surface $Z=E(1) \sharp 8 \overline{\mathbb{CP}^2}$}

The elliptic fibration $E(1)$ is the one in Section~\ref{section:Godeaux}. In the construction of $Z$, we will use the
sections $e_4$, $e_7$, $e_8$ among the four sections of $E(1)$. We
denote the sections $e_4$, $e_7$, $e_8$ by $S_1$, $S_2$, $S_3$,
respectively. We first blow up at the two nodes of the nodal
singular fibers $F_1$ and $F_2$ so that we obtain a blown-up
rational elliptic surface $W=E(1) \sharp
2\overline{\mathbb{CP}^2}$; Figure~\ref{figure:W-2}. Let $E_1$ and
$E_2$ be the exceptional curves over the nodes of $F_1$ and $F_2$,
respectively. We further blow up at each two marked points
$\bullet$ and blow up four times at the marked point $\bigodot$ in
Figure~\ref{figure:W-2}. We then get a rational surface $Z=E(1)
\sharp 8 \overline{\mathbb{CP}^2}$; Figure~\ref{figure:Z-2}. There
exists one linear chain of $\mathbb{CP}^1$s in $Z$ whose dual
graph is
\begin{equation*}
\udc{-6}{F_1}-\udc{-1}{S_3}-\udc{-8}{F_2}-\udc{-2}{S_1}-\udc{-2}{\ell}-\udc{-3}{S_2}-\udc{-2}{E_5}-\udc{-2}{E_6}-\udc{-2}{E_7}.
\end{equation*}
Notice that the $(-1)$-curve $S_3$ is contracted in the way down,
which fixes the configuration so that we obtain one singular point of class $T$ whose resolution graph is given by
\begin{equation*}
C_{24,5}:=\uc{-5}-\uc{-7} -\uc{-2}-\uc{-2}-\uc{-3}-\uc{-2}-\uc{-2}-\uc{-2}.
\end{equation*}

The divisor $B_1+B_2+B_3+B_4$ is the $2$-divisible one as before.
The $C_{24,5}$ and the $(-2)$-curves $B_3$, $B_4$, $B_1+B_6+B_2$,
and $M$ are contracted to obtain a singular surface $Y'$. One can
use again Theorem \ref{theorem:Burns-Wahl} to show that the space
$\Defq(Y')$ is smooth (of dimension $8$).

\begin{figure}[tbh]
\centering
\includegraphics[scale=1]{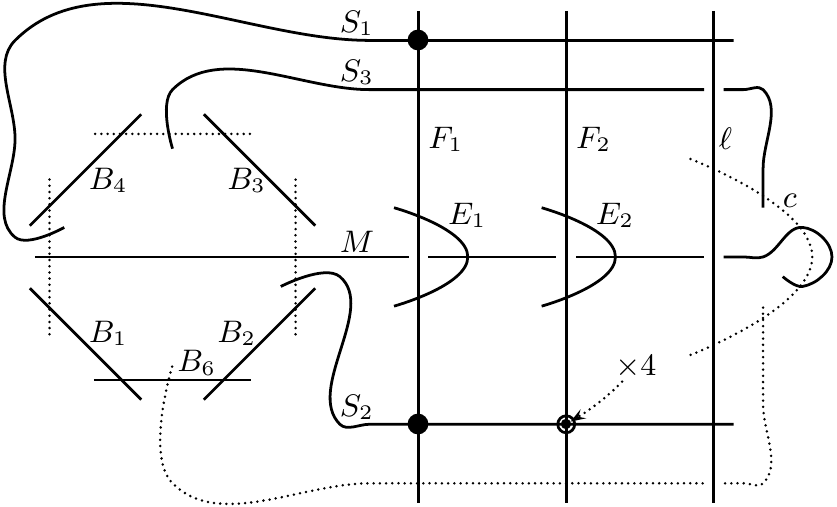}
\caption{A blown-up rational elliptic surface $W=E(1) \sharp 2\overline{\mathbb{CP}^2}$}
\label{figure:W-2}
\end{figure}

\begin{figure}[tbh]
\centering
\includegraphics[scale=1]{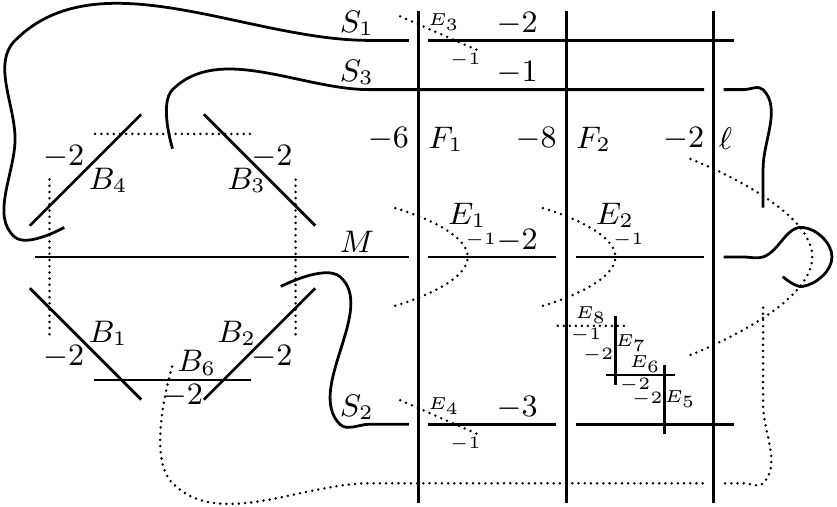}
\caption{A rational surface $Z = E(1) \sharp 8\overline{\mathbb{CP}^2}$}
\label{figure:Z-2}
\end{figure}


\providecommand{\bysame}{\leavevmode\hbox to3em{\hrulefill}\thinspace}

\end{document}